\documentclass[12pt]{article}
\usepackage[utf8]{inputenc}
\usepackage{amsmath}
\usepackage{amssymb}
\usepackage{mathrsfs}
\usepackage{amsthm}
\usepackage{cite}
\usepackage{hyperref}
\usepackage{enumerate}
\usepackage[left=0.7in,right=0.7in,top=0.7in,bottom=0.7in]{geometry}
\usepackage{titlesec}
\usepackage{xcolor} 
\titleformat*{\section}{\large\bfseries}
\newtheorem{theorem}{Theorem}[section]
\newtheorem{lemma}[theorem]{Lemma}
\newtheorem{corollary}[theorem]{Corollary}

\newtheorem{proposition}[theorem]{Proposition}

\numberwithin{equation}{section}
\title{Generalized Weighted Composition Operators on Vector-Valued Weighted Bergman Space}
\author{\large Anuradha Gupta and Geeta Yadav$^*$}
\date{}
\begin{document}
\maketitle
\begin{abstract} 

In this research article the necessary and sufficient conditions for the norm of composition operator $C_{\Phi}$ on $\mathcal{A}_{\alpha}^2(H)$ to be one are obtained. Moreover, $C_{\Phi}$ is unitary on $\mathcal{A}_{\alpha}^2(H)$ if and only if it is co-isometry. The necessary and sufficient condition for Hermitian and normal composition operators on $\mathcal{A}_{\alpha}^2(H)$ are also explored. Also, the characterization for boundeness of generalized weighted composition operator is obtained under some condition on $\Phi.$


\end{abstract}    
\section{Introduction and Preliminaries}
Let $\mathbb{D}=\{z \in \mathbb{C} : \|z\|<1 \}$ denote the open unit disk.
Let $\mathbb{N}_{0}$ denote the set of non-negative integers. It can be easily seen that for a vector space $Y,$ the set of all $Y$-valued functions on $\mathbb{D},$ denoted by $V(\mathbb{D}, Y),$ is a vector space under the standard pointwise addition and scalar multiplication.

For a complex Banach space $X,$ let $\mathcal {H}(\mathbb{D}, X)$ denote the space of analytic $X$-valued functions $f$ on $\mathbb{D},$ that is 
$$f(z)=\sum_{n=0}^{\infty} y_n z^n, \,\,\, y_n\in X \,\, \text{for all}\,\,z \in \mathbb{D} $$ 

For $\alpha >-1$ and $1\leq p<\infty$ the vector-valued weighted Bergman space $\mathcal{A}_{\alpha}^p(X)$ is defined by
\begin{equation*}
\mathcal{A}_{\alpha}^p(X)=\{f \in \mathcal{H}(\mathbb{D}, X):||f||^{p}_{\mathcal{A}_{\alpha}^p(X)}=  \int_{\mathbb{D}} ||f(z)||^p dA_{\alpha}(z) < \infty\}
\end{equation*}
where $dA_{\alpha}(z)= (\alpha+1)(1-|z|^2)^{\alpha} dA(z)$ and $dA$ denote the normalized area measure on $\mathbb{D}.$  In particular, for  $\alpha=0$ it is denoted by $\mathcal{A}^p(X)$ and  is called the vector-valued Bergman space. For $X=\mathbb{C},$ $\mathcal{A}_{\alpha}^p(X)$ is denoted by $\mathcal{A}_{\alpha}^p,$ weighted Bergman space.

Throughout the paper let $H$ denote a seperable complex Hilbert space with an orthonormal basis $\{e_{n}\}_{n \in \mathbb{N}_{0}}.$ Then, for $p=2,$ $\mathcal{A}_{\alpha}^2(H)$ is a Hilbert space  with the inner product defined by 

$$\langle f, g \rangle_{\mathcal{A}_{\alpha}^2(H)}= \int_{\mathbb{D}} \langle  f(z), g(z) \rangle  \, dA_{\alpha}(z) \; \text{for all} \; f,g \in \mathcal{A}_{\alpha}^2(H).$$
Equivalently, for $f(z)=\sum_{n=0}^{\infty} y_n z^n,$ $g(z)=\sum_{n=0}^{\infty} s_n z^n$ in $\mathcal{A}_{\alpha}^2(H)$  
$$\langle f, g \rangle_{\mathcal{A}_{\alpha}^2(H)}=\sum_{n=0}^{\infty} \frac{n! \, \Gamma(2+\alpha)}{\Gamma(n+2+\alpha)} \langle y_n,s_n \rangle,  $$
and
\begin{align*}
\mathcal{A}_{\alpha}^2(H)=\Big\{f(z)=\sum_{n=0}^{\infty} y_n z^n  \in \mathcal {H}(\mathbb{D},H): \,||f||_{\mathcal{A}_{\alpha}^2(H)}^2= \sum_{n=0}^{\infty} \frac{n! \, \Gamma(2+\alpha)}{\Gamma(n+2+\alpha)} ||y_n||^2 < \infty \Big\}.
\end{align*}

In this paper we will frequently use that for $n\in \mathbb{N}_{0}$ \cite{Book operator theory in function spaces}
\begin{equation}
\int_{\mathbb{D}} |z|^{2n} \, dA_{\alpha}(z)=  \frac{n! \, \Gamma(2+\alpha)}{\Gamma(n+2+\alpha)}.
\end{equation}

Let $V(\mathbb{D}, H)$ be the set of all $H$-valued functions and $W \subseteq V(\mathbb{D}, H)$ be a Hilbert space. Then, $W$ is said to be $H$-valued reproducing kernel Hilbert space on $\mathbb{D}$ if for each $z\in \mathbb{D}$ the linear map $L_{z}: W \longrightarrow H \; $ defined by 
$$ L_{z}(f)=f(z) \; \text{for all} \; f\in W,$$
is bounded. For more information on vector-valued reproducing kernel Hilbert space one can refer to \cite{Book on vector val rep kernel}.\\

Let $Z$ be the space of analytic functions defined on $\mathbb{D}$ and let $\Phi$ and $\Psi$ be analytic functions on $\mathbb{D}$ such that $\Phi$ is a self-map on $\mathbb{D}.$ Then, the weighted composition operator, $C_{\Psi,\Phi},$ induced by $\Phi$ and $\Psi$ on $Z$ is defined by
$$C_{\Psi,\Phi}f=\Psi \cdot (f  \circ \Phi) \;\, \text{for all}\, f\in Z. $$
If $\Psi \equiv 1,$ then $C_{\Psi,\Phi}$ is the composition operator $C_{\Phi}.$

Researchers find it intriguing to establish the relation between the function theoretic behaviours of $\Phi$ and $\Psi$ with the operator theoretic properties of weighted composition operator $C_{\Psi,\Phi}.$ Many researchers have studied these relationships on different Hilbert spaces, for example, Hardy space and Bergman space on the unit disc.\\
Sharma and Bhanu \cite{Imp Vector valued Hardy} obtained fundamental results for vector-valued Hardy Hilbert space and also proved many important results related to composition operators; like characterization for co-isometry, unitary, normal and Hermitian composition operators. For basics on vector-valued Bergman spaces one can refer to  \cite{Bergman and Bloch sp of vector valued} \cite{Intro to vector valued Bergman spaces} \cite{V Imp basics vector valued bergman}. Recently, Guo and Wang \cite{2022 paper} obtained complete characterization for the boundedness of difference of two weighted composition operators from weak to strong weighted vector-valued Bergman spaces.

This paper is organized in the following manner. In section 2, we obtain some basic characterization for $\mathcal{A}_{\alpha}^2(H)$ space including its reproducing kernel function. In section 3, we proved that for an analytic self-map $\Phi$ on $\mathbb{D},$ the norm of composition operator $C_{\Phi}$ on $\mathcal{A}_{\alpha}^2(H)$ is one if and only if $\Phi(0)=0.$ Further, we obtain that $C_{\Phi}$ is co-isometry if and only if $\Phi$ is a rotation on $\mathbb{D}$. Moreover, we showed that $C_{\Phi}$ is unitary if and only if it is co-isometry. In section 4 we obtain the necessary and sufficient conditions for adjoint operator of a composition operator on $\mathcal{A}_{\alpha}^2(H)$ to be some composition operator. We also completely characterized the Hermitian and normal composition operators on $\mathcal{A}_{\alpha}^2(H)$. In the last section we discuss the boundedness of generalized weighted composition operators on $\mathcal{A}_{\alpha}^2(H).$


 \section{The Space $\mathcal{A}_{\alpha}^2(H)$ }
In 1999, Sharma and Bhanu \cite{Imp Vector valued Hardy}  discussed some important characterizations like orthonormal basis and  reproducing kernel for vector-valued Hardy space. In 2003, orthonormal basis (\cite{Bergman and Bloch sp of vector valued}, \cite{V Imp basics vector valued bergman}) and  reproducing kernel \cite{V Imp basics vector valued bergman}  for vector-valued Bergman space were obtained. Motivated by \cite{Bergman and Bloch sp of vector valued}, \cite{Imp Vector valued Hardy} and \cite{V Imp basics vector valued bergman} we discuss some known and unkown characterizations of vector-valued weighted Bergman space $\mathcal{A}_{\alpha}^2(H).$ \\    
For $m,n \in \mathbb{N}_{0},$ function $\mathcal{E}_{m,n}: \mathbb{D} \longrightarrow H$ defined by
$$\mathcal{E}_{m,n}(z)=d_{m\alpha} \,z^m\, e_{n}, \; e_{n}\in H, $$
is in $\mathcal{A}_{\alpha}^2(H)$ where $d_{m\alpha}=\sqrt{\dfrac{\Gamma(m+2+\alpha)}{m! \, \Gamma(2+\alpha)}}.$ 

\begin{proposition} \label{orthonormal basis for vector val weig Bergman sp}
Set $\{\mathcal{E}_{m,n}: m,n \in \mathbb{N}_{0} \}$ forms an orthonormal basis for $\mathcal{A}_{\alpha}^2(H).$ 
\end{proposition}
\begin{proof}
Consider
\begin{align*}
\langle \mathcal{E}_{m,n},\mathcal{E}_{s,t} \rangle_{\mathcal{A}_{\alpha}^2(H)} &=d_{m\alpha} d_{s\alpha}  \int_{\mathbb{D}} \langle  z^m e_{n}, z^s e_{t} \rangle  \, dA_{\alpha}(z) \\
                  &= d_{m\alpha} d_{s\alpha} \langle  e_{n},  e_{t} \rangle \int_{\mathbb{D}} z^m \bar{z}^s  \, dA_{\alpha}(z) \\
                         &= \begin{cases}
                         0  & \text{if}\; m \neq s \; \text{or} \; n \neq t \\
                         1  & \text{if}\; m= s \; \text{and}\; n=t
                         \end{cases}
\end{align*} 
Thus, $\{\mathcal{E}_{m,n}: m,n \in \mathbb{N}_{0} \}$ is an orthonormal set. To prove the required result it is sufficient to prove that $Span\{\mathcal{E}_{m,n}: m,n \in \mathbb{N}_{0} \}^{\perp}=\{0\}.$ Let $f(z)=\sum_{p=0}^{\infty} y_p z^p \in Span\{\mathcal{E}_{m,n}: m,n \in \mathbb{N}_{0} \}^{\perp}.$ Now, for $m,n \in \mathbb{N}_{0} $ 
\begin{align*}
\langle f,\mathcal{E}_{m,n} \rangle_{\mathcal{A}_{\alpha}^2(H)} 
                        &=\sum_{p=0}^{\infty} \int_{\mathbb{D}} \langle  y_{p}z^{p}, \mathcal{E}_{m,n}(z) \rangle  \, dA_{\alpha}(z) \\
                         &=d_{m\alpha}  \sum_{p=0}^{\infty} \int_{\mathbb{D}} z^{p} \bar{z}^{m} \langle  y_{p},e_{n} \rangle  \, dA_{\alpha}(z) \\
                         &=d_{m\alpha} \langle  y_{m},e_{n} \rangle  \int_{\mathbb{D}} z^{m} \bar{z}^{m}   \, dA_{\alpha}(z).
\end{align*}
Thus,  $\langle f,\mathcal{E}_{m,n} \rangle_{\mathcal{A}_{\alpha}^2(H)}=0$ for all $ m,n \in \mathbb{N}_{0},$ gives   $\langle  y_{m},e_{n} \rangle=0$ for all $ m,n \in \mathbb{N}_{0}.$ Since $\{e_{n} \}_{n \in  \mathbb{N}_{0}}$ is an orthonormal basis for Hilbert space $H,$ therefore, $y_{m}=0$ for all $ m \in \mathbb{N}_{0}$ and consequently $f \equiv 0.$
\end{proof}
If $f(z)=\sum_{n=0}^{\infty} y_n z^n  \in \mathcal{A}_{\alpha}^2(H),$ then its norm can be expressed in the following form:
\begin{equation} \label{norm in terms of dm alpha} 
||f||_{\mathcal{A}_{\alpha}^2(H)}^2=\sum_{n=0}^{\infty} \frac{1}{d^{2}_{n\alpha}} ||y_n||^2, 
\end{equation}
which will be used frequently in the subsequent result.
 

The following theorem ($n=1, p=2$) is obtained by Oliver (\cite{R V Oliver Thesis}, Theorem 2.1.1). However we have given an alternative proof for the same on the space $\mathcal{A}_{\alpha}^2(H)$ in the following: 
\begin{theorem} \label{fz norm less than f intosome norm} 
Let $\alpha >-1.$  Then, 
$$ ||f(z)|| \leq  \frac{||f||_{\mathcal{A}_{\alpha}^2(H)}}{\sqrt{(1-|z|^2)^{\alpha+2}}},$$
for any $f \in \mathcal{A}_{\alpha}^2(H)$ and $z \in \mathbb{D}.$
\end{theorem}
\begin{proof}
 Let $f(z)=\sum_{n=0}^{\infty} y_n z^n \in \mathcal{A}_{\alpha}^2(H).$ Then  
\begin{align*}
||f(z)|| &= ||\sum_{n=0}^{\infty} y_n z^n||  \\
          & \leq \sum_{n=0}^{\infty} ||y_n|| |z^n| \\
          &\leq \Big( \sum_{n=0}^{\infty}||y_n||^2 \frac{1}{d^{2}_{n\alpha}} \Big)^{1/2} \Big( \sum_{n=0}^{\infty} d^{2}_{n\alpha} |z|^{2n}\Big)^{1/2} \\
          &= \frac{||f||_{\mathcal{A}_{\alpha}^2(H)}}{\sqrt{(1-|z|^2)^{\alpha+2}}}.
\end{align*}
\end{proof}
From Theorem \ref{fz norm less than f intosome norm}, it clearly follows that for each $z\in \mathbb{D}$ the point-evaluation map  $L_{z}: \mathcal{A}_{\alpha}^2(H) \longrightarrow H,$ defined by 
$$ L_{z}(f)=f(z) \; \text{for all} \; f\in \mathcal{A}_{\alpha}^2(H),$$
is bounded and thus $ \mathcal{A}_{\alpha}^2(H)$ is a reproducing kernel Hilbert space.

Although the reproducing Kernel for the  vector-valued weighted Bergman space where the underlying Hilbert space need not be seperable, is known. For the completness we are discussing the reproducing kernel for the $H$-valued weighted Bergman spaces where $H$ is seperable.

From Theorem \ref{fz norm less than f intosome norm}, it also follows that for  $j\in \mathbb{N}_{0}$ and $ z\in \mathbb{D}$ the function  $\Lambda_{z}^{j} :\mathcal{A}_{\alpha}^2(H) \longrightarrow \mathbb{C}$ 
defined by 
\begin{equation} \label{for reproducing inner product eq 1}
\Lambda_{z}^{j}(f)=\langle f(z), e_{j} \rangle,
\end{equation}
 is a bounded linear functional. Therefore, by Riesz Representation Theorem it follows that there exists $K_{z}^{j} \in \mathcal{A}_{\alpha}^2(H)$ and satisfies 
\begin{equation} \label{for reproducing inner product eq 2} 
 \Lambda_{z}^{j}(f)=\langle f, K_{z}^{j} \rangle_{\mathcal{A}_{\alpha}^2(H)} \; \text{for all} \; f \in \mathcal{A}_{\alpha}^2(H). 
\end{equation}
Function $K_{z}^{j}$ is known as the reproducing kernel. 
From equations \eqref{for reproducing inner product eq 1} and \eqref{for reproducing inner product eq 2} it follows that 
\begin{equation} \label{for reproducing inner product eq 3} 
\Lambda_{z}^{j}(f\circ \Phi)=\langle (f \circ \Phi)(z), e_{j} \rangle= \langle f (\Phi (z)), e_{j} \rangle=\Lambda_{\Phi(z)}^{j}(f).
\end{equation}

\begin{lemma}  \label{reproducing kernel defined}
 The reproducing kernel function $K_{z}^{j}$ for $z \in \mathbb{D}$ and $ j \in \mathbb{N}_{0}$ is 
\begin{equation*}  \label{equation reproducing kernel defined}
K_{z}^{j}(w)=\frac{e_{j}}{(1-w\bar{z})^{2+\alpha}} 
\end{equation*}
\end{lemma}
\begin{proof} From Proposition \ref{orthonormal basis for vector val weig Bergman sp}, it follows that we can write $K_{z}^{j} \in \mathcal{A}_{\alpha}^2(H) $ as
$K_{z}^{j}=\sum_{m,n \in \mathbb{N}_{0}} \langle K_{z}^{j},\mathcal{E}_{m,n} \rangle_{\mathcal{A}_{\alpha}^2(H)} \mathcal{E}_{m,n}$
\begin{align*}
K_{z}^{j}(w) 
             &=\sum_{m,n \in \mathbb{N}_{0}} \overline{\langle \mathcal{E}_{m,n}(z), e_{j} \rangle} \mathcal{E}_{m,n}(w)  \\
             &=\sum_{m,n \in \mathbb{N}_{0}} d_{m\alpha}\bar{z}^m \langle e_{j},e_{n} \rangle \mathcal{E}_{m,n}(w) \\   
             &=\sum_{m \in \mathbb{N}_{0}} d^{2}_{m\alpha}\bar{z}^m w^{m} e_{j} \\ 
             &=\frac{e_{j} }{(1-w\bar{z})^{2+\alpha}}.                   
\end{align*}
\end{proof}
From Lemma \ref{reproducing kernel defined} and equations \eqref{for reproducing inner product eq 1}, \eqref{for reproducing inner product eq 2}  it follows that for $z \in \mathbb{D}, j \in \mathbb{N}_{0}$ 
$$||K_{z}^{j}||^2= \langle K_{z}^{j}, K_{z}^{j} \rangle_{\mathcal{A}_{\alpha}^2(H)}= \langle K_{z}^{j}(z), e_{j} \rangle=\langle \frac{e_{j}}{(1-|z|^2)^{2+\alpha}}, e_{j} \rangle=\frac{1}{(1-|z|^2)^{2+\alpha}}.$$
Thus,
\begin{equation} \label{norm of reproducing kernel} 
||K_{z}^{j}||=\frac{1}{\sqrt{(1-|z|^2)^{2+\alpha}}}.
\end{equation}
\begin{lemma}  \label{lemma linear span of kernel}  
Set  $Span\{K_{z}^{j}:j\in \mathbb{N}_{0},z\in \mathbb{D} \}$ is dense in $\mathcal{A}_{\alpha}^2(H).$
\end{lemma}
\begin{proof}
Let $f\in (Span\{K_{z}^{j}:j\in \mathbb{N}_{0},z\in \mathbb{D} \})^{\perp}.$ Then, from equations \eqref{for reproducing inner product eq 1} and \eqref{for reproducing inner product eq 2} we get that  
\begin{align*}
\langle f,K_{z}^{j} \rangle=0 \Rightarrow  \langle f(z),e_{j} \rangle=0 \;\text{for all}\, z\in \mathbb{D} \, \text{and}\, j\in \mathbb{N}_{0}.
\end{align*}
Since $\{e_{n}\}_{n \in \mathbb{N}_{0}}$ is an orthonormal basis for $H,$ it follows that  $f(z)=0$ for all $z\in \mathbb{D}$ and hence, $f\equiv 0.$ Therefore, by \cite{Kreyszig} the required result follows.
\end{proof}
\section{Co-isometry and Unitary Composition Operator on $\mathcal{A}_{\alpha}^2(H)$}
In this section we completely characterize the co-isometry composition operator on $\mathcal{A}_{\alpha}^2(H).$ We have also discussed that composition operator is co-isometry if and only if it is unitary. We know that a bounded linear operator $T$ on a Hilbert space $Y$ is an isometry if $T^{*}T=I_{Y}$ and co-isometry if $T^{*}$ is an isometry, that is, $TT^{*}=I_{Y}.$ Further, $T$ is unitary if both $T$ and $T^{*}$ are isometries on $Y.$ \\

It is well known that for $c\in \mathbb{D},$ if function $\Phi_{c}:\mathbb{D} \longrightarrow \mathbb{D}$ is defined by
$$\Phi_{c}(z)=\frac{c-z}{1- \bar{c}z} \; \text{for all} \; z\in \mathbb{D},$$
then $\Phi\in Aut(\mathbb{D}),$ set of all analytic self map on $\mathbb{D}$ that are bijective. Also, 
 $\Phi_{c}^{-1}=\Phi_{c}$ and $\Phi_{c}'(z)=-\dfrac{1-|c|^2}{(1-\bar{c}z)^2}.$\\
The following important result is well known for composition operators on $\mathcal{A}_{\alpha}^p$ \cite{Book operator theory in function spaces}, scalar-valued weighted Bergman space. Motivated by the same result we prove it for the vector-valued Bergman space, $\mathcal{A}_{\alpha}^2(H)$. In the following theorem we have obtained an upper bound for the norm of composition operator on $\mathcal{A}_{\alpha}^2(H).$
\begin{theorem} \label{upper bound norm of comp op in term of induced fun}
Let $\Phi$ be an analytic self-map on $\mathbb{D}.$ Then, 
$$||C_{\Phi}|| \leq \sqrt{\left(\frac{1+|\Phi(0)|}{1-|\Phi(0)|}\right)^{2+\alpha}},$$
and consequently $C_{\Phi}$ is bounded on $\mathcal{A}_{\alpha}^2(H).$ 
\end{theorem}
\begin{proof}
Let $c=\Phi(0)$ and $\Psi$ be an analytic self-map on $\mathbb{D}$ defined by
\begin{equation} \label{comp op upper bound eq*}
\Psi(z)=(\Phi_{c} \circ \Phi)(z) \; \text{for all} \; z \in \mathbb{D}.
\end{equation}
 Then $\Psi(0)=\Phi_{c}(c)=0$ and applying $\Phi_{c}^{-1}$ on both sides of equation \eqref{comp op upper bound eq*} we get $(\Phi_{c} \circ \Psi)(z) =\Phi(z).$ 
Since $z \longrightarrow ||(f \circ \Phi_{c})(z)||^2$  is subharmonic \cite{book subharmonic result} for $f \in \mathcal{A}_{\alpha}^2(H),$ therefore, by Littlewood Subordination Theorem for $r\in (0,1)$ we have
\begin{align*}
\int_{0}^{2\pi} ||(f \circ \Phi)(re^{i\theta})||^{2} d\theta &=\int_{0}^{2\pi} ||(f \circ \Phi_{c}) (\Psi(re^{i\theta}))||^2 d\theta \\
                                                               & \leq \int_{0}^{2\pi}||(f \circ \Phi_{c}) (re^{i\theta})||^2 d\theta. 
\end{align*}
This implies 
$$ \int_{\mathbb{D}} ||(f \circ \Phi)(z)||^2 dA_{\alpha}(z) \leq  \int_{\mathbb{D}} ||(f \circ \Phi_{c})(z)||^2 \, dA_{\alpha}(z).$$
Applying change of variable in the integral on right side we get
$$ \int_{\mathbb{D}} ||(f \circ \Phi)(z)||^2 \, dA_{\alpha}(z) \leq \int_{\mathbb{D}} ||f(z)||^2  \frac{(1-|c|^2)^{2+\alpha}}{|1-\bar{c}z|^{2(2+\alpha)}} \, dA_{\alpha}(z). $$
For $c, z \in \mathbb{D}$ it can be easily seen that $\dfrac{1}{|1-\bar{c}z|} \leq \dfrac{1}{1-|c|}.$ Using this in above equation we get
\begin{align*}
\int_{\mathbb{D}} ||(f \circ \Phi)(z)||^2 \,dA(z) & \leq \frac{(1-|c|^2)^{2+\alpha}}{(1-|c|)^{2(2+\alpha)}} \int_{\mathbb{D}} {||f(z)||^2} \, dA_{\alpha}(z) \\
                                          &=\left(\frac{1+|\Phi(0)|}{1-|\Phi(0)|}\right)^{2+\alpha} \int_{\mathbb{D}} {||f(z)||^2} \, dA_{\alpha}(z).  
\end{align*}
Thus, 
$$||C_{\Phi}f||_{\mathcal{A}_{\alpha}^2(H)}^2 \leq \left(\frac{1+|\Phi(0)|}{1-|\Phi(0)|}\right)^{2+\alpha} ||f||_{\mathcal{A}_{\alpha}^2(H)}^2 \;\; \text{for all} \; f \in \mathcal{A}_{\alpha}^2(H).$$ 
Hence,
 $$||C_{\Phi}|| \leq  \sqrt{\left(\frac{1+|\Phi(0)|}{1-|\Phi(0)|}\right)^{2+\alpha}}. $$      
\end{proof}
\begin{corollary} \label{if phi zero is zero then comp op norm one}
Let $\Phi$ be an analytic self-map on $\mathbb{D}$ such that $\Phi(0)=0.$ Then, $||C_{\Phi}||=1.$
\end{corollary}
\begin{proof}
Define constant function $g \equiv e_{n} \in \mathcal{A}_{\alpha}^2(H)$ for $n \in \mathbb{N},$ then  
$$ ||C_{\Phi}|| \geq \frac{||C_{\Phi}g||_{\mathcal{A}_{\alpha}^2(H)}}{||g||_{\mathcal{A}_{\alpha}^2(H)}}=\frac{||g||_{\mathcal{A}_{\alpha}^2(H)}}{||g||_{\mathcal{A}_{\alpha}^2(H)}}=1. $$ 
Taking $\Phi(0)=0$ in Theorem \ref{upper bound norm of comp op in term of induced fun} we get that $ ||C_{\Phi}|| \leq 1.$  Hence, $||C_{\Phi}||=1.$
\end{proof}  
Recall that the adjoint operator, $T^{*},$ of a bounded linear operator $T$ on a Hilbert space $Z$ is defined by 
 $$\langle Tx,y \rangle= \langle x,T^{*}y \rangle \;\, \text{for all} \; x,y \in Z .$$

Now, for $z \in \mathbb{D}$ and $j\in \mathbb{N}_{0}$
\begin{equation} \label{Cphi star Kz value of function}
\langle f, C_{\Phi}^{*}K_{z}^{j} \rangle_{\mathcal{A}_{\alpha}^2(H)} =\langle C_{\Phi}f, K_{z}^{j} \rangle_{\mathcal{A}_{\alpha}^2(H)}= \langle f \circ \Phi, K_{z}^{j} \rangle_{\mathcal{A}_{\alpha}^2(H)} \;\, \text{for all} \; f \in \mathcal{A}_{\alpha}^2(H).
\end{equation}
Combining equations \eqref{for reproducing inner product eq 2}, \eqref{for reproducing inner product eq 3} and \eqref{Cphi star Kz value of function} we get
$$\langle f, C_{\Phi}^{*}K_{z}^{j} \rangle_{\mathcal{A}_{\alpha}^2(H)}=\Lambda_{z}^{j}(f \circ \Phi)=\Lambda_{\Phi(z)}^{j}(f)=\langle f,K_{\Phi(z)}^{j} \rangle_{\mathcal{A}_{\alpha}^2(H)} \;\, \text{for all} \; f \in \mathcal{A}_{\alpha}^2(H). $$
Hence, 
\begin{equation} \label{adjoint of comp on kernel}
C_{\Phi}^{*}K_{z}^{j}=K_{\Phi(z)}^{j}.
\end{equation}
The following result is motivated by (\cite{Martinez and Rosenthal}, Theorem 5.1.10) result on Hardy space and the proof for  $\mathcal{A}_{\alpha}^2(H)$ space follows on the similar lines.   
\begin{theorem} \label{theorem lower bound for compo op}
Let $\Phi$ be an analytic self-map on $\mathbb{D}.$ Then, 
\begin{equation} \label{lower bound for norm of comp op}
 \frac{1}{\sqrt{(1-|\Phi(0)|^2)^{2+\alpha}}} \leq ||C_{\Phi}||,
 \end{equation}
 for $C_{\Phi}$ on $\mathcal{A}_{\alpha}^2(H).$
\end{theorem}
The following result gives the bounds of composition operator on $\mathcal{A}_{\alpha}^2(H)$ which follows directly from Theorem \ref{upper bound norm of comp op in term of induced fun} and Theorem \ref{theorem lower bound for compo op}.
\begin{corollary}
Let $\Phi$ be an analytic self-map on $\mathbb{D}.$ Then,
\begin{equation*} 
\frac{1}{\sqrt{(1-|\Phi(0)|^2)^{2+\alpha}}} \leq ||C_{\Phi}|| \leq \sqrt{\left(\frac{1+|\Phi(0)|}{1-|\Phi(0)|}\right)^{2+\alpha}}
 \end{equation*}
\end{corollary}
The following result gives the necessary condition for an isometric composition operator on $\mathcal{A}_{\alpha}^2(H).$ 
\begin{corollary} \label{comp op norm one implies phi zero is zero}
Let $\Phi$ be an analytic self-map on $\mathbb{D}$  and  $C_{\Phi}$ be an isometry on $\mathcal{A}_{\alpha}^2(H)$. Then, $\Phi(0)=0$
\end{corollary}
\begin{proof}
Let $C_{\Phi}$ be an isometry. Then, $||C_{\Phi}||=1$ and from equation \eqref{lower bound for norm of comp op} it follows that
$$\frac{1}{\sqrt{(1-|\Phi(0)|^2)^{2+\alpha}}}\leq 1 \Rightarrow   1-|\Phi(0)|^2 \geq 1. $$
Thus, $\Phi(0)=0.$
\end{proof}
As a direct consequence of Corollary \ref{if phi zero is zero then comp op norm one} and  Corollary \ref{comp op norm one implies phi zero is zero}, the condition for the unit norm of composition operator is obtained.
\begin{proposition}
Let $\Phi$ be an analytic self-map on $\mathbb{D}.$ Then for composition operator $C_{\Phi}$ on $\mathcal{A}_{\alpha}^2(H),$  $||C_{\Phi}||=1$ if and only if  $\Phi(0)=0.$ 
\end{proposition} 

The following Lemma will be instrumental in several results in subsequent sections. 
\begin{lemma} \label{adjoint of composition when phi is alpha z}
Let $\Phi$ and $\Psi$ be analytic self-map on $\mathbb{D}$ such that $\Phi(z)=\lambda z$ and $\Psi(z)=\bar{\lambda} z$ with $|\lambda|\leq 1.$ Then, the adjoint of composition operator $C_{\Phi}$ on $\mathcal{A}_{\alpha}^2(H)$ is a composition operator $C_{\Psi},$ that is, $ C_{\Phi}^{*}=C_{\Psi}.$
\end{lemma}
\begin{proof} 
In view of Theorem \ref{lemma linear span of kernel}, to prove the required result, it is sufficient to prove that for $j \in \mathbb{N}_{0}$ and $z\in \mathbb{D}$
$$C_{\Phi}^{*}K_{z}^{j}=C_{\Psi}K_{z}^{j}.$$
Now, for $j \in \mathbb{N}_{0}$ and $z\in \mathbb{D}$
$$K_{\Phi(z)}^{j}(w)=\frac{e_{j}}{(1-w\, \overline{\Phi(z)})^{2+\alpha}}=\frac{e_{j}}{(1-(\bar{\lambda}w)\bar{z} )^2}=\frac{e_{j}}{(1-\Psi(w)\bar{z} )^{2+\alpha}}=K_{z}^{j}(\Psi(w)).$$
Combining above equation with equation \eqref{adjoint of comp on kernel} we get
$$(C_{\Phi}^{*}K_{z}^{j})(w)=K_{\Phi(z)}^{j}(w)=K_{z}^{j}(\Psi(w))=(C_{\Psi}K_{z}^{j})(w) \;\, \text{for all} \; j \in \mathbb{N}_{0} \; \text{and} \; z,w\in \mathbb{D}.  $$
Thus, $C_{\Phi}^{*}K_{z}^{j}=C_{\Psi}K_{z}^{j}$ and hence, $ C_{\Phi}^{*}=C_{\Psi}.$
\end{proof}
\begin{theorem} \label{if phi is rotation then comp op is isometry} 
Let $\Phi$ be an analytic self-map on $\mathbb{D}$ such that it is a rotation, that is, $\Phi(z)=\lambda z$ and $|\lambda|=1.$ Then, $C_{\Phi}$ is an isometry on $\mathcal{A}_{\alpha}^2(H).$ 
\end{theorem} 
\begin{proof}
Let $f(z)=\sum_{n=0}^{\infty} y_n z^n \in \mathcal{A}_{\alpha}^2(H).$ Then, 
$$(C_{\Phi}f)(z)= f(\Phi(z))=f(\lambda z)=\sum_{n=0}^{\infty} \lambda^n y_n z^n.$$
So
$$||C_{\Phi}f||_{\mathcal{A}_{\alpha}^2(H)}^2=\sum_{n=0}^{\infty} \frac{1}{d^{2}_{n\alpha}} |\lambda|^{2n}||y_n||^2=\sum_{n=0}^{\infty} \frac{1}{d^{2}_{n\alpha}}||y_n||^2=||f||_{\mathcal{A}_{\alpha}^2(H)}^2$$
Thus, $C_{\Phi}$ is an isometry.
\end{proof} 

Following result characterized co-isometric composition operators on $\mathcal{A}_{\alpha}^2(H).$ 
\begin{theorem} \label{com op is coisometry iff phi is rotation}
Let $\Phi$ be an analytic self-map on $\mathbb{D}.$ Then, $C_{\Phi}$ is co-isometry on $\mathcal{A}_{\alpha}^2(H)$ if and only if $\Phi$ is a rotation, that is, $\Phi(z)=\lambda z$ for some $\lambda \in \mathbb{C}$ with $|\lambda|=1.$ 
\end{theorem}
\begin{proof}
Let $C_{\Phi}$ be a co-isometry. Then, $||C_{\Phi}^{*}f||_{\mathcal{A}_{\alpha}^2(H)}=||f||_{\mathcal{A}_{\alpha}^2(H)}$ for all $f \in \mathcal{A}_{\alpha}^2(H).$
This implies for $j \in \mathbb{N}_{0}$ and $z\in \mathbb{D}$ 
\begin{align*}
||C_{\Phi}^{*}K_{z}^{j}||_{\mathcal{A}_{\alpha}^2(H)} &=||K_{z}^{j}||_{\mathcal{A}_{\alpha}^2(H)} \\
\Rightarrow ||K_{\Phi(z)}^{j}||_{\mathcal{A}_{\alpha}^2(H)} &=||K_{z}^{j}||_{\mathcal{A}_{\alpha}^2(H)} \\
\Rightarrow \frac{1}{\sqrt{(1-|\Phi(z)|^2)^{2+\alpha}}} &=\frac{1}{\sqrt{(1-|z|^2)^{2+\alpha}}} 
\end{align*}
Thus, $|\Phi(z)|=|z|$ for all $z\in \mathbb{D}$ and hence, $\Phi(z)=\lambda z$ for some $\lambda \in \mathbb{C}$ with $|\lambda|=1.$\\
Conversely, let $\Phi(z)=\lambda z$ with $|\lambda|=1.$ Then, by Lemma \ref{adjoint of composition when phi is alpha z} $C_{\Phi}^{*}=C_{\Psi}$ where $\Psi(z)=\overline{\lambda} z.$ Further, by Theorem \ref{if phi is rotation then comp op is isometry} it follows that $C_{\Phi}^{*}$ is an isometry. Hence, $C_{\Phi}$ is co-isometry. 
\end{proof}
\begin{theorem} \label{com op is unitary iff phi is rotation}
Let $\Phi$ be an analytic self-map on $\mathbb{D}.$ Then, $C_{\Phi}$ is unitary on $\mathcal{A}_{\alpha}^2(H)$ if and only if $\Phi(z)=\lambda z$ where $|\lambda|=1.$
\end{theorem}
\begin{proof}
Let $C_{\Phi}$ be unitary. Then, $C_{\Phi}$ is co-isometry and by Theorem \ref{com op is coisometry iff phi is rotation}, $\Phi(z)=\lambda z$ with $|\lambda|=1.$ \\
Conversely, let $\Phi(z)=\lambda z$ with $|\lambda|=1.$ Then, by Theorem \ref{com op is coisometry iff phi is rotation} $C_{\Phi}$ is co-isometry and by Theorem \ref{if phi is rotation then comp op is isometry}, $C_{\Phi}$ is isometry. Hence, $C_{\Phi}$ is unitary.
\end{proof}
The following result directly follows from Theorem \ref{com op is coisometry iff phi is rotation} and Theorem \ref{com op is unitary iff phi is rotation}. 
\begin{corollary}
Let $\Phi$ be an analytic self-map on $\mathbb{D}.$ Then, $C_{\Phi}$ is co-isometry on $\mathcal{A}_{\alpha}^2(H)$ if and only if it is unitary.  
\end{corollary}
\section{Hermitian and Normal Composition Operators}
In 1999, Sharma and Bhanu \cite{Imp Vector valued Hardy} obtained complete characterization for some composition operators on vector-valued Hardy space like normal and Hermitian composition operators. With this motivation we obtain the necessary and sufficient condition for adjoint of composition operator on vector-valued weighted Bergman space, $\mathcal{A}_{\alpha}^2(H),$ to be some composition operator. Further, we also have obtained characterization for Hermitian and normal composition operator on $\mathcal{A}_{\alpha}^2(H)$. \\
Recall that a bounded linear operator $T$ on a Hilbert space $Y$ is Hermitian if $T^{*}=T$ and $T$ is normal if $T^{*}T=TT^{*}.$

\begin{theorem}\label{necessary and sufficient condition for adjoint of comp op}
Let $\Phi$ be an analytic self-map on $\mathbb{D}.$ Then, the adjoint of composition operator $C_{\Phi}$ on $\mathcal{A}_{\alpha}^2(H)$ is some composition operator if and only if $\Phi(z)=\lambda z$ with $|\lambda|\leq 1.$
\end{theorem}
\begin{proof} Let $C_{\Phi}^{*}=C_{\Psi}$ for some analytic self-map $\Psi$ on $\mathbb{D}.$ Let  $\Phi(z)=\sum_{n=0}^{\infty} a_{n} z^{n}$ and $\Psi(z)=\sum_{n=0}^{\infty} b_{n} z^{n}$ where $a_{n}$ and $b_{n} \in \mathbb{C}$ for $n\in \mathbb{N}_{0}.$ Since $C_{\Phi}^{*}=C_{\Psi}$ and $(C_{\Psi}K_{0}^{j})(w)=K_{0}^{j}(\Psi(w))=e_{j}$ for all $w\in \mathbb{D},$ therefore, 
\begin{align*}
||C_{\Phi}^{*}K_{0}^{j}||_{\mathcal{A}_{\alpha}^2(H)}^2 &=||C_{\Psi}K_{0}^{j}||_{\mathcal{A}_{\alpha}^2(H)}^2 \\
\Rightarrow || K_{\Phi(0)}^{j}||_{\mathcal{A}_{\alpha}^2(H)}^2 &=||C_{\Psi}K_{0}^{j}||_{\mathcal{A}_{\alpha}^2(H)}^2 \\
\Rightarrow \frac{1}{(1-|\Phi(0)|^2)^{2+\alpha}} &=1.  
\end{align*} 
Thus, $\Phi(0)=0,$ that is, $a_{0}=0.$ Similarly, we can prove $b_{0}=0$ using the fact that $C_{\Phi}=C_{\Psi}^{*}.$ So $\Phi(z)=\sum_{n=1}^{\infty} a_{n} z^{n}$ and $\Psi(z)=\sum_{n=1}^{\infty} b_{n} z^{n}.$ Now
$$C_{\Psi}\mathcal{E}_{1,1}(z)=\mathcal{E}_{1,1}(\Psi(z))=\sqrt{2+\alpha} \,\Psi(z) e_{1} \; \text{and} \;  \frac{1}{d_{m\alpha}}= \sqrt{\frac{m! \, \Gamma(2+\alpha)}{\Gamma(m+2+\alpha)} }.$$
For $m \geq 1$ 
\begin{align}
 \langle \frac{1}{\sqrt{2+\alpha}} C_{\Psi} \mathcal{E}_{1,1}, d_{m\alpha} \mathcal{E}_{m,1} \rangle_{\mathcal{A}_{\alpha}^2(H)} =&\frac{1}{\sqrt{2+\alpha}} d_{m\alpha} \langle C_{\Psi} \mathcal{E}_{1,1}, \mathcal{E}_{m,1} \rangle_{\mathcal{A}_{\alpha}^2(H)} \notag \\
                         &=\frac{1}{\sqrt{2+\alpha}} d_{m\alpha} \langle  \mathcal{E}_{1,1} \circ \Psi, \mathcal{E}_{m,1}\rangle_{\mathcal{A}_{\alpha}^2(H)} \label{adjoint of comp op is comp op eq 1}
\end{align}
Using polar coordinates form it is easy to see that
\begin{align}
\int_{\mathbb{D}} \langle  \mathcal{E}_{1,1}( \Psi(z)), \mathcal{E}_{m,1}(z) \rangle \,   dA_{\alpha}(z)&= \sqrt{2+\alpha} \, d_{m\alpha}\, \frac{1}{\pi} \int_{0}^{1} \int_{0}^{2\pi} \langle  \Psi(re^{i\theta})e_{1},(re^{i\theta})^{m} e_{1} \rangle \,(\alpha+1)(1-r^2)^{\alpha}\,  r dr d\theta \notag\\
                          &= \sqrt{2+\alpha} \, d_{m\alpha}\,\, \frac{1}{\pi} \sum_{n=1}^{\infty} \int_{0}^{1} \int_{0}^{2\pi}   b_{n} r^{n+m} e^{i(n-m)\theta} \langle e_{1},e_{1} \rangle \, (\alpha+1)(1-r^2)^{\alpha}\,  r dr d\theta  \notag\\
                          &= \sqrt{2+\alpha} \, d_{m\alpha}\,\, \frac{1}{d^2_{m\alpha}}b_{m}.\label{adjoint of comp op is comp op eq 2} 
\end{align}
Combining \eqref{adjoint of comp op is comp op eq 1} and \eqref{adjoint of comp op is comp op eq 2}, for $m \geq 1$ we get 
\begin{equation} \label{adjoint of comp op is comp op eq 3} 
b_{m}=\langle \frac{1}{\sqrt{2+\alpha}} C_{\Psi} \mathcal{E}_{1,1}, d_{m\alpha} \mathcal{E}_{m,1} \rangle_{\mathcal{A}_{\alpha}^2(H)}.
\end{equation}
Also, for $m \geq 1$
\begin{align*}
\langle \frac{1}{\sqrt{2+\alpha}} C_{\Psi} \mathcal{E}_{1,1}, d_{m\alpha} \mathcal{E}_{m,1} \rangle_{\mathcal{A}_{\alpha}^2(H)} =&\frac{1}{\sqrt{2+\alpha}} d_{m\alpha} \langle  \mathcal{E}_{1,1}, C_{\Psi}^{*} \mathcal{E}_{m,1} \rangle_{\mathcal{A}_{\alpha}^2(H)}. 
\end{align*}
Since $C_{\Phi}^{*}=C_{\Psi},$ therefore, $C_{\Phi}=C_{\Psi}^{*}.$ Using this in above equation we get 
\begin{align*}
\langle \frac{1}{\sqrt{2+\alpha}} C_{\Psi} \mathcal{E}_{1,1}, d_{m\alpha} \mathcal{E}_{m,1} \rangle_{\mathcal{A}_{\alpha}^2(H)} &=\frac{1}{\sqrt{2+\alpha}} d_{m\alpha} \langle  \mathcal{E}_{1,1}, C_{\Phi} \mathcal{E}_{m,1} \rangle_{\mathcal{A}_{\alpha}^2(H)} \\
                  &=d^2_{m\alpha}   \int_{\mathbb{D}} \langle ze_{1}, (\Phi(z))^m e_{1} \rangle \, dA_{\alpha}(z) \\
                         &=d^2_{m\alpha} \frac{1}{\pi}  \int_{0}^{1} \int_{0}^{2\pi}   \langle re^{i\theta} e_{1},(\sum_{n=1}^{\infty} a_{n}(re^{i\theta})^n)^m e_{1} \rangle \, (\alpha+1)(1-r^2)^{\alpha}\, r dr d\theta. 
\end{align*}
On simplifying this we get  
\begin{align} \label{adjoint of comp op is comp op eq 4} 
\langle \frac{1}{\sqrt{2+\alpha}} C_{\Psi} \mathcal{E}_{1,1}, d_{m\alpha} \mathcal{E}_{m,1} \rangle_{\mathcal{A}_{\alpha}^2(H)}=\overline{a_{1}} \, \delta_{m1}
\end{align}
Equations \eqref{adjoint of comp op is comp op eq 3} and \eqref{adjoint of comp op is comp op eq 4} together implies that
\begin{equation*}
b_{m}=\overline{a_{1}} \, \delta_{m1} \;\, \text{for} \, m\geq 1. 
\end{equation*}
This implies $\Psi(z)=\overline{a_{1}} \, z.$ Since $\Psi$ is a self-map on $\mathbb{D}$ therefore $|a_{1}| \leq 1.$ Further, combining the fact that $C_{\Phi}=C_{\Psi}^{*}$ with Lemma \ref{adjoint of composition when phi is alpha z} we get 
\begin{equation*}
\Phi(z)= \overline{\overline{a_{1}}} z= a_{1} z \,\,\text{with} \,\, |a_{1}| \leq 1.
\end{equation*}
Conversely, let $\Phi(z)=\lambda z$ with $|\lambda| \leq 1.$ Then, by Lemma \ref{adjoint of composition when phi is alpha z},  $C_{\Phi}^{*}=C_{\Psi}$ where $\Psi(z)=\bar{\lambda} z.$ So the adjoint of composition operator is again a composition operator.
\end{proof}
As a consequence of the above Theorem we have the following result:
\begin{corollary}
Let $\Phi$ be an analytic self-map on $\mathbb{D}.$ Then, $C_{\Phi}$ is Hermitian on $\mathcal{A}_{\alpha}^2(H)$ if and only if $\Phi(z)=\lambda z$ where $|\lambda| \leq 1$ and $\lambda$ is a real number.
\end{corollary}
\begin{proof}
Let $C_{\Phi}$ be Hermitian. Then, by Theorem \ref{necessary and sufficient condition for adjoint of comp op}  $\Phi(z)=\lambda z$ with $|\lambda| \leq 1.$ It follows by Lemma \ref{adjoint of composition when phi is alpha z}, that $C_{\Phi}^{*}=C_{\Psi}$ for $\Psi(z)=\overline{\lambda} z.$  Since $C_{\Phi}^{*}=C_{\Phi},$ therefore, $C_{\Psi}=C_{\Phi}.$ Let $g(z)=z e_{0} \in \mathcal{A}_{\alpha}^2(H),$ then 
$$(C_{\Phi}g)(z)=(C_{\Psi}g)(z) \Rightarrow g(\Phi(z))=g(\Psi(z)) \Rightarrow (\lambda z-\overline{\lambda}z)e_{0}=0 \,\, \text{for all} \, z \in \mathbb{D}.$$
Thus, $\lambda= \overline{\lambda},$ that is, $\lambda $ is a real number.\\
Conversely, let $\Phi(z)=\lambda z$ where $|\lambda| \leq 1$ and $\lambda$ is a real number and let $\Psi(z)=\overline{\lambda} z.$  Then, $\Psi(z)=\lambda z= \Phi(z)$ and by Lemma \ref{adjoint of composition when phi is alpha z} it follows that $C_{\Phi}^{*}=C_{\Phi}.$ Hence, $C_{\Phi}$ is Hermitian.
\end{proof}
The necessary and sufficient conditions for a composition operator to be normal operator on $\mathcal{A}_{\alpha}^2(H)$ are obtained in the following: 
\begin{theorem} \label{Comp op is normal iff condition}
Let $\Phi$ be analytic self-map on $\mathbb{D}.$ Then, $C_{\Phi}$ is normal on $\mathcal{A}_{\alpha}^2(H)$ if and only if $\Phi(z)=\lambda z$ with $|\lambda| \leq 1$ 
\end{theorem}
\begin{proof}
Let $\Phi(z)=\sum_{l=0}^{\infty} a_{l} z^{l}$ where $a_{l} \in \mathbb{C}$ and let $C_{\Phi}$ be normal. Then, for $f \in \mathcal{A}_{\alpha}^2(H)$  
$$||C_{\Phi}^{*} f||_{\mathcal{A}_{\alpha}^2(H)}^2= \langle C_{\Phi}^{*} f, C_{\Phi}^{*} f \rangle_{\mathcal{A}_{\alpha}^2(H)}=\langle C_{\Phi} C_{\Phi}^{*} f, f \rangle_{\mathcal{A}_{\alpha}^2(H)}=\langle C_{\Phi}^{*} C_{\Phi} f, f \rangle_{\mathcal{A}_{\alpha}^2(H)}= ||C_{\Phi} f ||_{\mathcal{A}_{\alpha}^2(H)}^2. \; $$
Thus,
\begin{equation} \label{normal adjoint of comp and comp eq}
||C_{\Phi}^{*} f||_{\mathcal{A}_{\alpha}^2(H)}^2=||C_{\Phi} f ||_{\mathcal{A}_{\alpha}^2(H)}^2 \;\, \text{for all}\,\,f \in \mathcal{A}_{\alpha}^2(H).
\end{equation} 
Taking $f=K_{0}^{j},\; j\in \mathbb{N}_{0}$ in above equation we get
\begin{align*}
||C_{\Phi}^{*} K_{0}^{j} ||_{\mathcal{A}_{\alpha}^2(H)}^2 &=||C_{\Phi} K_{0}^{j} ||_{\mathcal{A}_{\alpha}^2(H)}^2 \\
\Rightarrow \frac{1}{(1-|\Phi(0)|^2)^{2+\alpha}} &=1.   
\end{align*}
Hence, $\Phi(0)=0$ implies $a_{0}=0.$ So, $\Phi(z)=\sum_{l=1}^{\infty} a_{l} z^{l}.$  Since $\{\mathcal{E}_{mn}: m,n \in \mathbb{N}_{0} \}$ forms an orthonormal basis for $\mathcal{A}_{\alpha}^2(H),$ we can write 
$$C_{\Phi}^{*} \mathcal{E}_{1,1}=\sum_{m,n=0}^{\infty} \langle C_{\Phi}^{*} \mathcal{E}_{1,1}, \mathcal{E}_{m,n} \rangle_{\mathcal{A}_{\alpha}^2(H)} \mathcal{E}_{m,n}. $$
Consider
\begin{align}
||C_{\Phi}^{*} \mathcal{E}_{1,1}||_{\mathcal{A}_{\alpha}^2(H)}^2 &=\sum_{m,n=0}^{\infty} |\langle C_{\Phi}^{*} \mathcal{E}_{11}, \mathcal{E}_{m,n} \rangle_{\mathcal{A}_{\alpha}^2(H)}|^2 \notag\\
                      &=\sum_{m,n=0}^{\infty} |\langle  \mathcal{E}_{1,1}, C_{\Phi} \mathcal{E}_{m,n} \rangle_{\mathcal{A}_{\alpha}^2(H)}|^{2} \notag\\
                      &=\sum_{m,n=0}^{\infty} |\langle  \mathcal{E}_{1,1},  \mathcal{E}_{m,n}\circ \Phi \rangle_{\mathcal{A}_{\alpha}^2(H)}|^{2}. \label{ normal eq 1}
\end{align}
For $m,n \in \mathbb{N}_{0}$
\begin{align}
\langle  \mathcal{E}_{1,1},  \mathcal{E}_{m,n}\circ \Phi \rangle_{\mathcal{A}_{\alpha}^2(H)} &= \sqrt{2+\alpha}\, d_{m\alpha} \int_{\mathbb{D}} \langle z e_{1}, (\Phi(z))^{m} e_{n} \rangle \, dA_{\alpha}(z) \notag\\
 &=\sqrt{2+\alpha}\, d_{m\alpha} \frac{1}{\pi}  \int_{0}^{1} \int_{0}^{2\pi}   \langle re^{i\theta} e_{1},(\sum_{l=1}^{\infty} a_{l}(re^{i\theta})^l)^m e_{n} \rangle \,(\alpha+1)(1-r^2)^{\alpha}\, r dr d\theta \notag\\
&=\begin{cases} 
 (2+\alpha)\,\langle e_{1}, e_{n} \rangle \overline{a_{1}}\, \int_{\mathbb{D}} |z|^2 \, dA_{\alpha}(z) & \text{if}\,\, m= 1 \\
 0 & \text{if}\,\, m\neq 1 
\end{cases}\notag \\
 &=\overline{a_{1}} \, \langle e_{1},e_{n}\rangle  \delta_{m,1}. \label{ normal eq 2} 
\end{align}
Combining equations \eqref{ normal eq 1} and \eqref{ normal eq 2} we get
\begin{align}
||C_{\Phi}^{*} \mathcal{E}_{1,1}||_{\mathcal{A}_{\alpha}^2(H)}^2 &=\sum_{m,n=0}^{\infty} |\overline{a_{1}} \, \langle e_{1},e_{n}\rangle  \delta_{m,1}|^2 \notag\\
         &=\sum_{n=0}^{\infty} |\overline{a_{1}} \, \langle e_{1},e_{n}\rangle |^2 \notag\\
         &= |a_{1}|^{2}. \label{ normal eq 3} 
\end{align}
From equation \eqref{normal adjoint of comp and comp eq} we get 
 \begin{align}
||C_{\Phi}^{*} \mathcal{E}_{1,1}||_{\mathcal{A}_{\alpha}^2(H)}^2 &=||C_{\Phi} \mathcal{E}_{1,1}||_{\mathcal{A}_{\alpha}^2(H)}^2 \notag\\
 &=(\sqrt{2+\alpha})^2\int_{\mathbb{D}}   \langle \Phi(z)e_{1},\Phi(z)e_{1} \rangle \, dA_{\alpha}(z) \notag\\                
 &=(2+\alpha) \sum_{l=1}^{\infty} \int_{\mathbb{D}}  |a_{l}|^2 |z|^{2l} \, dA_{\alpha}(z) \notag\\ 
 &=(2+\alpha) \sum_{l=1}^{\infty} |a_{l}|^2 \frac{l! \, \Gamma(2+\alpha)}{\Gamma(l+2+\alpha)} \notag \\
 &= \sum_{l=1}^{\infty} |a_{l}|^2 \frac{l! \, \Gamma(3+\alpha)}{\Gamma(l+2+\alpha)}. \label{ normal eq 4} 
 \end{align}
Equating \eqref{ normal eq 3} and \eqref{ normal eq 4}  we get that $a_{l}=0$ for all $l \geq 2$ and hence $\Phi(z)=a_{1} z.$ Since $\Phi$ is a self-map on $\mathbb{D},$ therefore, $|a_{1}| \leq 1$.\\
Conversely, let $\Phi(z)=\lambda z$ with $|\lambda| \leq 1.$ Then, by Lemma \ref{adjoint of composition when phi is alpha z}  $C_{\Phi}^{*}=C_{\Psi}$ for $\Psi(z)=\bar{\lambda} z.$ For all $f\in\mathcal{A}_{\alpha}^2(H), z \in \mathbb{D}$  
$$(C_{\Phi}C_{\Phi}^{*}f)(z)=(C_{\Phi}C_{\Psi}f)(z)=f(|\lambda|^{2}z)$$
and
$$(C_{\Phi}^{*}C_{\Phi}f)(z)=(C_{\Psi}C_{\Phi}f)(z)=f(|\lambda|^{2}z).$$
Therefore, $C_{\Phi}$ is normal on $\mathcal{A}_{\alpha}^2(H)$.
\end{proof}

\section{Boundedness of Generalized Weighted Composition Operator}

Let $\Phi$ and $\Psi$ be analytic function on $\mathbb{D}$ such that $\Phi$ is a self-map on $\mathbb{D}.$ For $r\in \mathbb{N}_{0},$ the generalized weighted composition operator (\cite{Zhu Generalized weighted 1}, \cite{Zhu Generalized weighted 2}) on $\mathcal{A}_{\alpha}^2(H)$ denoted by $D_{\Phi,\Psi}^{r}$ be defined as
$$D_{\Phi,\Psi}^{r}f=\Psi \cdot (f^{(r)}  \circ \Phi) \;\; \text{for all} \; f \in \mathcal{A}_{\alpha}^2(H),$$
where $D$ is the differentiation operator and $D^{r}f=f^{(r)}$ for f in $\mathcal{A}_{\alpha}^2(H)$. For $r=0,$ $D_{\Phi,\Psi}^{r}$ is the weighted composition operator $C_{\Psi,\Phi}.$

By $L^2$ we denote the space of all complex valued measurable functions on $\mathbb{D}$ such that 
$$||h||_2^{2}=\int_{\mathbb{D}} |h(z)|^2 dA_{\alpha}(z) < \infty .$$

We all know that differentiation operator is unbounded. In 2017, Kumar and Abbas \cite{Pawan and Zaheer} obtained complete characterization of operator $D_{\Phi,\Psi}^{r}$ for $r=1$ to be bounded on weighted Hardy spaces. In this section we will discuss the boundedness of $D_{\Phi,\Psi}^{r}$ for $r \geq 1$ on $\mathcal{A}_{\alpha}^2(H).$ The motivation of this section comes from \cite{Pawan and Zaheer}.

\begin{theorem} \label{boundedness of generalized weighted comp op}
Let $\Phi$ and $\Psi$ be analytic function on $\mathbb{D}$ such that $\Phi$ is a self-map on $\mathbb{D}$ and $\{\Phi^{n}\}_{n \in \mathbb{N}_{0}}$ be orthogonal. Then, for $r \geq 1$ the generalized weighted composition operator $D_{\Phi,\Psi}^{r}$ is bounded on $\mathcal{A}_{\alpha}^2(H)$ if and only if there exist some number $K$ such that
$$||\Psi \cdot \Phi^{m-r}||_2 \leq \frac{K}{d_{m\alpha} \, m(m-1)(m-2) \ldots (m-(r-1))} \,\, \text{for all}\,\, m\geq r$$ 
\end{theorem}
\begin{proof}
Recall that for $m\in  \mathbb{N}_{0},$  $\mathcal{E}_{m,0}(z)=d_{m\alpha} z^{m} e_{0} \in \mathcal{A}_{\alpha}^2(H).$ For $m \geq r$
$$\mathcal{E}_{m,0}^{(r)}=d_{m\alpha} m(m-1)(m-2) \ldots (m-(r-1))z^{m-r} e_{0}$$
Let $D_{\Phi,\Psi}^{r}$ be bounded for $r\geq 1.$ Then, there exists a number $K$ which satisfy 
$$|| D_{\Phi,\Psi}^{r}f ||_{\mathcal{A}_{\alpha}^2(H)} \leq K ||f||_{\mathcal{A}_{\alpha}^2(H)} \;\; \text{for all} \; f \in \mathcal{A}_{\alpha}^2(H).$$
From this it follows that   
$$|| D_{\Phi,\Psi}^{r} \,\mathcal{E}_{m,0} ||_{\mathcal{A}_{\alpha}^2(H)} \leq K ||\mathcal{E}_{m,0}||_{\mathcal{A}_{\alpha}^2(H)} $$
Since $||\mathcal{E}_{m,0}||=1,$ therefore, for $m \geq r$
$$|| m(m-1)(m-2) \ldots (m-(r-1)) d_{m\alpha} \Psi \cdot (e_{0}\Phi^{m-r})||_{\mathcal{A}_{\alpha}^2(H)} \leq K $$
Since $||e_{0}||=1,$ above reduces to 
$$||\Psi \cdot \Phi^{m-r}||_{2} \leq \frac {K}{d_{m\alpha} \, m(m-1)(m-2) \ldots (m-(r-1))} \,\,\; \text{for all}\,\, m\geq r$$
Conversely, suppose 
$$||\Psi \cdot \Phi^{m-r}||_{2} \leq \frac{K}{d_{m\alpha} \, m(m-1)(m-2) \ldots (m-(r-1))} \,\, \text{for all}\,\, m\geq r$$ 
Let $f(z) =\sum_{m=0}^{\infty} y_{m} z^{m} \in \mathcal{A}_{\alpha}^2(H)$ where $y_{m}\in H.$ Then 
$$f^{(r)}(z)=\sum_{m=r}^{\infty} m(m-1)(m-2) \ldots (m-(r-1))\, y_{r}\, z^{m-r} $$
Now 
$$D_{\Phi,\Psi}^{r} f=\sum_{m=r}^{\infty} m(m-1)(m-2) \ldots (m-(r-1))\, y_{m}\, ( \Psi \cdot \Phi^{m-r}). $$
Since $\{\Phi^{n}\}_{n \in \mathbb{N}_{0}}$ is orthogonal it follows that 
\begin{align*}
&||D_{\Phi,\Psi}^{r} f||_{\mathcal{A}_{\alpha}^2(H)}^{2}\\ 
&=||\sum_{m=r}^{\infty} m(m-1)(m-2) \ldots (m-(r-1))\, y_{m}\, ( \Psi \cdot \Phi^{m-r}) ||_{\mathcal{A}_{\alpha}^2(H)}^{2} \\
   &=\sum_{m=r}^{\infty} m^{2}(m-1)^{2}(m-2)^{2} \ldots (m-(r-1))^{2}\, ||y_{m}||^{2}\, || \Psi \cdot \Phi^{m-r} ||_{2}^{2} \\
   & \leq \sum_{m=r}^{\infty} m^{2}(m-1)^{2}(m-2)^{2} \ldots (m-(r-1))^{2}\, ||y_{m}||^{2}\,\frac{K^{2}}{d_{m\alpha}^{2} \, m^{2}(m-1)^{2}(m-2)^{2} \ldots (m-(r-1))^{2}} \\
   &\leq \sum_{m=0}^{\infty} \frac{||y_{m}||^{2} }{d_{m\alpha}^{2}} K^{2}  
\end{align*} 
Thus,  $||D_{\Phi,\Psi}^{r} f||_{\mathcal{A}_{\alpha}^2(H)}  \leq K||f||_{\mathcal{A}_{\alpha}^2(H)}$ for all $f \in \mathcal{A}_{\alpha}^2(H).$ Hence $D_{\Phi,\Psi}^{r}$ is a bounded operator for $r \geq 1$.
\end{proof}
If $\Psi \equiv 1,$ then the generalized weighted composition operator $D_{\Phi,\Psi}^{r}=C_{\Phi}D^{r}$ and   
Theorem \ref{boundedness of generalized weighted comp op} reduces to the following:
\begin{corollary}
Let $\Phi$ be analytic function on $\mathbb{D}$ such that $\Phi$ is a self-map on $\mathbb{D}$ and $\{\Phi^{n}\}_{n \in \mathbb{N}_{0}}$ is orthogonal. Then, for $r \geq 1$ the operator $C_{\Phi}D^{r}$ is bounded on $\mathcal{A}_{\alpha}^2(H)$ if and only if there exist some number $K$ such that
$$|| \Phi^{m-r}||_{2} \leq \frac{K}{d_{m\alpha} \, m(m-1)(m-2) \ldots (m-(r-1))} \,\,\, \text{for all}\,\, m\geq r.$$ 
\end{corollary}

For $r=1$ and $\Psi=\Phi'$ derivative of $\Phi$  
$$ D_{\Phi,\Psi}^{r}f=\Phi' \cdot(f'\circ \Phi)  = DC_{\Phi}f, $$
 for $f \in \mathcal{A}_{\alpha}^2(H)$ and Theorem \ref{boundedness of generalized weighted comp op} reduces to the following:
\begin{corollary}
Let $\Phi$ be analytic function on $\mathbb{D}$ such that $\Phi$ is a self-map on $\mathbb{D}$ and $\{\Phi^{n}\}_{n \in \mathbb{N}_{0}}$ is orthogonal. Then, $DC_{\Phi}$ is bounded on $\mathcal{A}_{\alpha}^2(H)$ if and only if there exists some number $K$ such that
$$||\Phi' \cdot \Phi^{m-1}||_{2} \leq \frac{K}{d_{m\alpha} \, m} \,\,\, \text{for all}\,\, m\geq 1.$$ 
\end{corollary}



\section*{Declarations}

\begin{itemize}
\item \textbf{Funding}\\
Not applicable
\item \textbf{Competing interests}\\
Authors declare no conflicts of interest in this paper.
\end{itemize}

\textbf{Anuradha Gupta}\\
 Department of Mathematics, Delhi College of Arts and Commerce,\\
  University of Delhi, Netaji Nagar, \\
  New Delhi-110023, India.\\
  \vspace{0.2cm}
 email: dishna2@yahoo.in\\
   \textbf{Geeta Yadav} (Corresponding author)\\
  Department of Mathematics,\\
   University of Delhi, New Delhi-110007, India.\\
  email: ageetayadav@gmail.com
\end{document}